\documentclass{article}

\usepackage{indentfirst, bbm}
\usepackage{amsthm}
\usepackage{amsmath}
\usepackage{amssymb}
\usepackage{amsmath, amsthm, bbm}
\usepackage{framed, fullpage}

\newtheorem{theo}{Theorem}

\newtheorem{lemma}{Lemma}[section]

\newtheorem{prop}[lemma]{Proposition}

\newtheorem{claim}[lemma]{Claim}

\date{}

\title{A quantitative Lov\'asz criterion for Property B}

\author{Asaf Ferber \thanks{Department of Mathematics, MIT, Cambridge, MA, USA. Email: ferbera@mit.edu. Supported in part by NSF grant 6935855.}
\and Asaf Shapira \thanks{
School of Mathematics, Tel Aviv University, Tel Aviv 69978, Israel.
Email: asafico$@$tau.ac.il. Supported in part by ISF Grant 1028/16 and ERC Starting Grant 633509.}
}

\begin{document}

\maketitle
\begin{abstract}


A well known observation of Lov\'asz is that if a hypergraph is not $2$-colorable, then at least one pair of its edges intersect at a single vertex.
In this short paper we consider the quantitative version of Lov\'asz's criterion. That is, we ask how many pairs
of edges intersecting at a single vertex, should belong to a non $2$-colorable $n$-uniform hypergraph?
Our main result is an {\em exact} answer to this question, which further characterizes all the extremal hypergraphs.
The proof combines Bollob\'as's two families theorem with Pluhar's randomized coloring algorithm.

\end{abstract}

\section{Introduction}\label{sec:intro}

A hypergraph ${\cal H}=(V,E)$ consists of a vertex set $V$ and a set
of edges $E$ where each
$X \in E$ is a subset of $V$. If all edges of ${\cal H}$ have size $n$
then ${\cal H}$
is called an $n$-uniform hypergraph, or $n$-graph for short.
A hypergraph is $2$-colorable if one can assign each vertex $v\in V$
one of two colors, say $Red$/$Blue$, so that each
$X \in E$ contains vertices of both colors. Miller \cite{M}, and later Erd\H{o}s in various papers, referred
to this property as {\em Property B}, after F.~Bernstein \cite{B}
who introduced it in 1907. Since deciding if a hypergraph is
$2$-colorable is $NP$-hard one cannot hope to find a simple
characterization of all $2$-colorable hypergraphs. Instead, one
looks for general sufficient/necessary conditions for
having this property. For example, a famous result of Seymour \cite{S}
states that if ${\cal H}$ is not $2$-colorable then
$|E| \geq |V|$. Probably the most well studied question of this type
asks for the smallest number of edges in an $n$-graph that
is not $2$-colorable. The study of this quantity, denote $m(n)$, was
popularized by Erd\H{o}s, see \cite{AS} for a comprehensive treatment.
Despite much effort by many researchers, even the asymptotic value of $m(n)$ has not been determined yet.

A pair of edges $X,Y \in E({\cal H})$ is \emph{simple} if $|X\cap Y|=1$. Let $m_2({\cal H})$ denote the number of ordered simple pairs
of edges of ${\cal H}$. A well known observation of Lov\'asz \cite{L} states that if ${\cal H}$ is not $2$-colorable then $m_2({\cal H}) >0$. Despite its simplicity, this observation underlies the best known bounds for $m(n)$, see
\cite{CK,Pluhar}. It is natural to ask if one can obtain a quantitative version of Lov\'asz's observation, that is,
estimate how small can $m_2({\cal H})$ be in an $n$-graph not satisfying property $B$?
Our main result in this paper states that (somewhat surprisingly), one can give an exact answer to the above extremal question
as well as characterize the extremal $n$-graphs.

Let $K_{2n-1}^n$ denote the complete $n$-graph on $2n-1$ vertices. It is easy to see that $K_{2n-1}^n$ is not $2$-colorable
and that $m_2(K_{2n-1}^n)=n\cdot\binom{2n-1}{n}$. We first observe that this simple upper bound is tight.

\begin{prop}\label{prop:main}
If an $n$-graph is not $2$-colorable then $m_2({\cal H}) \geq n\cdot\binom{2n-1}{n}$.
\end{prop}

As with any extremal problem, one would like to know which graphs or hypergraphs are extremal with respect to this problem.
For example, Tur\'an's theorem states that among all $n$-vertex graphs not containing a complete
$t$-vertex subgraph, there is only one graph maximizing the number of edges. In the setting of
our problem, it is easy to see that $K_{2n-1}^n$ is not the only non $2$-colorable $n$-graph satisfying
$m_2({\cal H}) = n\cdot\binom{2n-1}{n}$, since one can take a copy of $K_{2n-1}^n$ and add to it more
vertices and edges without increasing the number of simple pairs. Our main result in this paper
characterizes the extremal $n$-graphs, by showing that this is in fact the only way to construct an $n$-graph meeting the bound of Proposition \ref{prop:main}.

\begin{theo}\label{thm:main}
If a non $2$-colorable $n$-graph ${\cal H}$ satisfies $m_2({\cal H}) = n\cdot\binom{2n-1}{n}$ then it contains a copy of $K_{2n-1}^n$.
\end{theo}

While the proof of Proposition \ref{prop:main} is implicit in Pluhar's \cite{Pluhar} argument for bounding $m(n)$,
the proof of Theorem \ref{thm:main} is more intricate,
relying on Bollob\'as's two families theorem \cite{Bo} as well as on a refined analysis of Pluhar's randomized algorithm for $2$-coloring $n$-graphs.

\section{Proof of Proposition \ref{prop:main}}

In this section we describe several preliminary observations regarding a coloring algorithm introduced in~\cite{Pluhar}, and use them to derive Proposition \ref{prop:main}. The algorithm is the following:



\bigskip

\noindent{\bf Algorithm }$\text{Col}({\cal H},\pi)$. The input is a hypergraph ${\cal H}=(V,E)$ and an ordering $\pi:V \mapsto \{1,\ldots,|V|\} $ (that is, $\pi$ is a bijection). The output is a $2$-coloring of $V$ (not necessarily a proper one). The algorithm runs in $|V|$ steps, where in each time step $1\leq i\leq |V|$, the vertex $\pi^{-1}(i)$ is being colored $Blue$ if this does not form any monochromatic $Blue$ edge. Otherwise, $\pi^{-1}(i)$ is colored $Red$.

\bigskip

We now state an important property of $\text{Col}({\cal H},\pi)$.
For two disjoint subsets $X,Y\subseteq V$, we use the notation $\pi(X)<\pi(Y)$ whenever $\max_{x \in X}\pi(x)<\min_{y \in Y}\pi(y)$, that is, the elements of $X$ precede all the elements of $Y$ in the ordering $\pi$.
Suppose $(X,Y)$ is a simple pair of edges in ${\cal H}$ with\footnote{Here, and in what follows, we slightly abuse notation by writing $y$ instead of the more appropriate $\{y\}$.} $X\cap Y=y$.
We say that $\pi$ {\em separates} $(X,Y)$ if $\pi(X\setminus y)<\pi(y)<\pi(Y\setminus y)$.

\begin{claim}\label{claim1}
If $\text{Col}({\cal H},\pi)$ fails to properly color ${\cal H}$ then $\pi$ separates at least one pair of simple edges.
\end{claim}

\begin{proof}
We first observe that (by definition) for every ordering $\pi$, the algorithm  $\text{Col}(H,\pi)$ does not produce monochromatic $Blue$ edges.
Suppose then it produced a $Red$ edge $Y\in E$. Let $y$ be the first vertex of $Y$ according to the ordering $\pi$.
If $y$ was colored red, then there must have been an edge $X$ so that $y\in X$, and all other vertices of $X$ were already colored $Blue$ (otherwise the algorithm would color $y$ $Blue$). This means $(X,Y)$ is simple and that $\pi$ separates it.
\end{proof}


Note that the claim above already shows that if ${\cal H}$ is not 2-colorable then $m_2({\cal H})>0$.
For the proof of Proposition \ref{prop:main} we will also need the following simple fact.

\begin{claim}\label{claim0}
A random permutation separates any given simple pair with probability $1/n\binom{2n-1}{n}$.
\end{claim}

\begin{proof}
Let $(X,Y)$ be a simple pair, and let $X\cap Y=y$. A permutation $\pi$ separates $(X,Y)$ if and only if $\pi(X\setminus y)<\pi(y)<\pi(Y\setminus y)$,
and this happens with probability exactly
$$\frac{(n-1)!(n-1)!}{(2n-1)!}=\frac{1}{n\binom{2n-1}{n}}$$
as desired.
\end{proof}





The above claims suffice for proving Proposition \ref{prop:main}.



\begin{proof}[Proof (of Proposition \ref{prop:main}):]
Assume $m_2({\cal H}) < n\binom{2n-1}{n}$. Suppose we pick a uniformly random $\pi$.
Then by the union bound and Claim \ref{claim0}, we infer that with positive probability $\pi$ does not separate any simple pair edges.
Hence, there is a $\pi$ not separating any simple pair. Claim \ref{claim1} then implies that $\text{Col}({\cal H},\pi)$ will produce a legal $2$-coloring of ${\cal H}$.
\end{proof}


\section{Proof of Theorem \ref{thm:main}}

For the rest of this section fix some non $2$-colorable $n$-graph ${\cal H}=(V,E)$ satisfying $m_2({\cal H})=n\binom{2n-1}{n}$.
We need to show that ${\cal H}$ contains a copy of $K^n_{2n-1}$. We start with a few preliminary claims regarding ${\cal H}$.

First, we show that no $\pi$ separates more than one simple pair.
\begin{claim}
  \label{claim: separates at most one pair}
  Every ordering $\pi$ separates at most one simple pair.
\end{claim}

\begin{proof}
Suppose $\pi$ separates two simple pairs. By Claim \ref{claim0}, the assumption on $m_2(\cal H)$, and by linearity of expectation, the expected number of simple pairs separated by a random permutation is exactly $1$. Hence, if $\pi$ separates $2$ simple pairs, then there must exist a permutation $\sigma$ which separates less than $1$, and therefore $0$, simple pairs. Therefore, by Claim \ref{claim1} we obtain that $\text{Col}({\cal H},\sigma)$ produces a legal $2$-coloring of ${\cal H}$, which is a contradiction to the assumption that $\mathcal H$ is not $2$-colorable.
\end{proof}

\begin{claim}\label{claim: no two edges intersect same vertex}
If $(X,Y)$ and $(X',Y)$ are simple pairs, then $X \cap Y \neq X' \cap Y$.
\end{claim}

\begin{proof}
We observe that if $X \cap Y = X' \cap Y=y$, then there is a $\pi$ that separates both $(X,Y)$ and $(X',Y)$, and this will contradict Claim \ref{claim: separates at most one pair}. Indeed, if $(X,Y)$ and $(X',Y)$ are simple pairs and $X\cap Y=X'\cap Y=y$,  then $(X\cup X')\setminus y$ and $Y$ are disjoint. Therefore,
any $\pi$ satisfying
$$
\pi((X\cup X')\setminus y)<\pi(y)<\pi(Y \setminus y)
$$
separates $(X,Y)$ and $(X',Y)$. This completes the proof.
\end{proof}

In addition to the above observations about ${\cal H}$, the last ingredient we will need is the following theorem of Bollob\'as \cite{Bo}.

\begin{lemma}\label{lem:bol}
Let $I$ be an index set. For all $i\in I$, let $A_i$ and $B_i$ be subsets of a set $V$ of $p$ elements satisfying the following conditions:
\begin{enumerate}
  \item $A_i\cap B_i=\emptyset$ for all $i\in I$, and
  \item $A_j\nsubseteq A_i\cup B_i$ for all $i\neq j\in I$.
\end{enumerate}
Then, we have
$$\sum_{i\in I}\frac{1}{\binom{p-|B_i|}{|A_i|}}\leq 1,$$
with equality if and only if $B_i=B$ for all $i\in I$ and the sets $A_i$ are all the $q$-tuples of the set $P\setminus B$ for some value of $q$.
\end{lemma}

Let us now show how to use Lemma \ref{lem:bol} in order to derive Theorem \ref{thm:main}. Recall that $V$ is the vertex set of ${\cal H}$ and set $p:=|V|$.
Let $M(\mathcal H)$ be a collection of simple pairs $(X,Y)$ defined as follows; out of all the simple pairs $(X,Y)$ with the same ``second'' set $Y$, put
in $M(\mathcal H)$ one of these pairs. Observe that by Claim \ref{claim: no two edges intersect same vertex} each $Y$ belongs to at most $|Y|=n$ simple pairs of the form $(X,Y)$ (i.e, with $Y$ as the second set), implying that $t:=|M(\mathcal H)|\geq \frac{1}{n} \cdot m_2(\mathcal H)= \binom{2n-1}{n}$.
We now define a collection $\mathcal F$ consisting of pairs of subsets of $V$ as follows: For every simple pair $s:=(X,Y)\in M(\mathcal H)$, define $A_s=X\setminus Y$ and $B_s=V\setminus (X\cup Y)$, and let $\mathcal F=\{(A_s,B_s): \text{ }s\in M(\mathcal H)\}$. For convenience, let
us rename the pairs in ${\cal F}$ as $(A_i,B_i)$ with $1 \leq i \leq t$.



Now we wish to show that $\mathcal F$ satisfies the conditions in Lemma \ref{lem:bol}. Observe that if it does, then since
$$\sum_{i=1}^t\frac{1}{\binom{p-|B_i|}{|A_i|}}=\sum_{i=1}^t \frac{1}{\binom{2n-1}{n-1}}\geq 1,$$
it follows by the first part of Lemma \ref{lem:bol} that the last inequality is in fact an equality. Therefore, by the second part of Lemma \ref{lem:bol}, we  conclude that all the $B_i$'s are the same set $B$, and the set of all the $A_i$'s consists of all $n-1$ subsets of a ground set of size $2n-1$. That is, let $B=B_i$ and $U=V\setminus B$. Then we have that $|U|=2n-1$, and that the sets $A_i$ are all the $n-1$ subsets of $U$. Since by construction we have that $U\setminus A_i\in E(\mathcal H)$ for all $i$, we conclude that $\mathcal H$ restricted to the set $U$ is a copy of $K^n_{2n-1}$ as desired. It thus remains to show the following:

\begin{claim}
$\mathcal F$ satisfies the conditions in Lemma \ref{lem:bol}
\end{claim}

\begin{proof}
The first condition $A_i\cap B_i=\emptyset$ for all $i$ is trivially satisfied by construction. For the second condition, let $(A,B)$ and $(A',B')$ be two elements in $\mathcal F$ coming from simple pairs $(X,Y)$ and $(X',Y')$ belonging to $M({\cal H})$, respectively. Recall that by the way we defined $M(\mathcal H)$ and $\mathcal F$ we have $Y\neq Y'$. Let us use $y$ and $y'$ to denote the unique elements in $X\cap Y$ and $X'\cap Y'$, respectively.
We wish to show that $A\nsubseteq A'\cup B'$, which, by construction, is implied by $(X\setminus y)\cap Y'\neq \emptyset$. Assuming $(X\setminus y)\cap Y'=\emptyset$, we will derive a contradiction to Claim \ref{claim: separates at most one pair} by showing that there is a permutation $\pi$ separating two distinct simple pairs.

Observe that it cannot be that $y\in Y'$. Indeed, if it was the case, then together with the assumption that $(X\setminus y)\cap Y'=\emptyset$ we would infer
that $(X,Y)$ and $(X,Y')$ are both simple pairs intersecting at $y$ (and distinct as $Y\neq Y'$), contradicting Claim \ref{claim: no two edges intersect same vertex}.
Assume then that $y\not \in Y'$ (so in particular $y \neq y'$). We claim that we can find a $\pi$ satisfying
$$
\pi(X\setminus y)<\pi(y)< \pi((X'\setminus y')\setminus X)<\pi(y')<\pi((Y\cup Y')\setminus (X \cup X')).
$$
Indeed, the only thing that needs to be justified is the ability to place $y'$ as above, which follows from the fact that $y' \in Y'$ and the
assumption $(X\setminus y)\cap Y'=\emptyset$ which together imply that $y' \not \in X$. Observe that since $\pi$ first places $X\setminus y$ and then $y$, the pair $(X,Y)$ is separated by $\pi$. Such a $\pi$ clearly places $X' \setminus y'$ before $y'$ and the assumption $(X\setminus y)\cap Y'=\emptyset$ together with the fact that $y \not \in Y'$ imply that such a $\pi$ places all of $Y' \setminus y'$ after $y'$, so it
separates $(X',Y')$ as well, giving us the desired contradiction.
\end{proof}

This completes the proof of Theorem \ref{thm:main}.

%
%
%



\end{document}